\documentclass[11pt]{amsart}

\newtheorem {theorem}{Theorem}
\newtheorem {lemma}{Lemma}

\newtheorem {corollary}{Corollary}

\newcommand{\C}{\mathbb C}
\newcommand{\CC}{\widehat{\C}}

\usepackage{amsfonts}
\usepackage{amssymb}
\usepackage{amsmath}
\usepackage{amsthm}
 \usepackage{enumerate}
\usepackage[all]{xy}
\usepackage{graphicx}
\usepackage{appendix}
   
\usepackage{url}
\usepackage{amscd}
\usepackage{texdraw}
\usepackage{graphicx}

\begin{document}

\title[Holomorphic Motions]{Winding Numbers And Full Extendibility in Holomorphic Motions}
 
\author{Yunping Jiang}
\address[Jiang]{Department of Mathematics\\
 Queens College of the City University of New York\\
 Flushing, NY 11367-1597\\
 and\\
 Department of Mathematics\\
 Graduate Center of the City University of New York;
 New York, NY 10016}
 \email[Jiang]{yunping.jiang@qc.cuny.edu}
 
 
\thanks{ 
This material is based upon work supported by the National Science Foundation. It is also partially supported by a collaboration grant from the Simons Foundation (grant number 523341) and PSC-CUNY awards.
}

\begin{abstract}
We construct an example of a holomorphic motion of a five-point subset of the Riemann sphere over an annulus such that it satisfies the zero winding number condition but is not fully extendable. 
\end{abstract}

\maketitle

\section{Introduction}

Suppose $\C$ is the complex plane and $\CC=\C\cup\{\infty\}$ is the Riemann sphere. Let
$$
\Delta=\{c\in \C\;|\; |c|<1\}
$$
be the open unit disk.  
Suppose $X$ is a connected complex manifold with a basepoint $t_{0}$. Let $E\subset \CC$ be a subset.
A map $\phi (t,z): X\times E\to \CC$ is called a holomorphic motion of $E$ over $X$ if 
\begin{itemize}
\item[i)] $\phi(t_{0}, z)=z$ for all $z\in E$;
\item[ii)]  for any fixed $t\in X$, $\phi_{t}(\cdot)= \phi (t, \cdot): E\to \CC$ is injective;
\item[iii)] for any fixed $z\in E$, $\phi^{z} (\cdot) =\phi (\cdot, z): X\to \CC$ is holomorphic.
\end{itemize} 
By pre- and post-compositing M\"obius transformations, without loss of generality, 
we can always assume that $E$ contains $0$, $1$, and $\infty$ and $\phi$ is normalized, 
that is, $\phi (t, z)=z$ for $z=0, 1$, and $\infty$ and all $t\in X$.  
Henceforth, all holomorphic motions in this paper are normalized. 

We say a holomorphic motion $\phi$ of $E$ over $X$ is {\em fully extendable} if there is 
a holomorphic motion $\psi$ of $\CC$ over $X$ such that the restriction
$$
\psi |X\times E=\phi.
$$
We call $\psi$ a full extension of $\phi$. 

Slodkowski's Theorem (see~\cite{Sl,GJW}) says that any holomorphic motion of $E$ over $\Delta$ is fully extendable. 
There is an example of a holomorphic motion of $E$ over a non-simply connected one-dimensional hyperbolic complex manifold $X$, which is not fully extendable  (see, for example,~\cite{E}). This leads to consider the following {\em zero winding number condition} in holomorphic motions of $E$ over a one-dimensional hyperbolic complex manifold $X$ as follows.

Suppose $\alpha$ is a closed curve in $\C$ not containing $0$. Then we can consider its winding number 
\begin{equation}~\label{windnum}
\eta= \eta (\alpha)= \frac{1}{2\pi} \oint_{\alpha}  d \arg \alpha.
\end{equation}
The winding number is an integer and $2\pi \eta$ is just the variation of argument on $\alpha$.
Given a holomorphic motion $\phi$ of $E$ over a one-dimensional hyperbolic complex manifold $X$ and given a simple closed curve $\beta$ in $X$ and a pair of points $z_{1}\not= z_{2}\in E$, we have the winding number $\eta(\alpha)$ of the closed curve 
$$
\alpha=\alpha (\beta, z_{1}, z_{2})=\phi (\beta, z_{1}) -\phi (\beta, z_{2})\subset \CC. 
$$
We say $\phi$ satisfies {\em the zero winding number condition} if $\eta (\alpha)=0$ for every simple closed curve $\beta$ in $X$ and every pair of points $z_{1}\not=z_{2}\in E$. 

Chirka gave a wonderful new proof of Slodkowski's Theorem in~\cite{Ch}. Further, he claimed in~\cite{Ch} that the zero winding number condition is a necessary and sufficient condition for a fully extendable holomorphic motion of $E$ over any connected one-dimensional hyperbolic complex manifold $X$. Since then, the following problem has been arisen in the study of holomorphic motions: Is the zero winding number condition a sufficient condition for a fully extendable holomorphic motion of $E$ over a connected one-dimensional hyperbolic complex manifold $X$? The purpose of this paper is to show that the zero winding number condition is not sufficient. 

\section{The Main Result}

Consider two arcs 
$$
T^{+}=\big\{ t=x+yi \in \C\;|\; |t+i/2|=\sqrt{5}/2, y\geq 0\big\}
$$ 
and
$$
T^{-} =\Big\{ \frac{1}{t} \;|\; t\in T^{+}\Big\}.
$$
Then $T=T^{+}\cup T^{-}$ is a simple closed curve such that $-1,1\in T$ (see Figure 1). It cuts $\C$ into two domains, one is bounded and simply connected domain $D$ and the other is unbounded and connected domain $\widetilde{D}$.
Then we have $0, -i\in D$ and $i\in \widetilde{D}$ (see Figure 1). 
Consider an annulus
$$
X =\{ t\in \C\;|\; dist (t, T) <\epsilon\}, \quad \epsilon >0,
$$
where $dist (t, T)$ means the distance between $t$ and $T$. Choose $\epsilon$ small enough such that $-2, 0, 1/2, 1/3, i, -i\not\in X$. 
Take $t_{0}=1$ as the basepoint of $X$ (see Figure 1).  

Let
$$
S^{+}=\Big\{ -\frac{1}{t} \;|\; t\in T^{+}\} \quad \hbox{and} \quad S^{-} =\Big\{ -\frac{1}{t} \;|\; t\in T^{-}\Big\}
$$
and $S=S^{+}\cup S^{-}$. It is again a simple closed curve such that $-1,1\in S$ and let
$$
Y=\{ t\in \C\;|\; dist (t, S) <\epsilon\}=\Big\{ -\frac{1}{t} \;|\; t\in X\Big\},
$$ 
be another annulus (see Figure 1).

\begin{figure}[h]
    \includegraphics[width=7in]{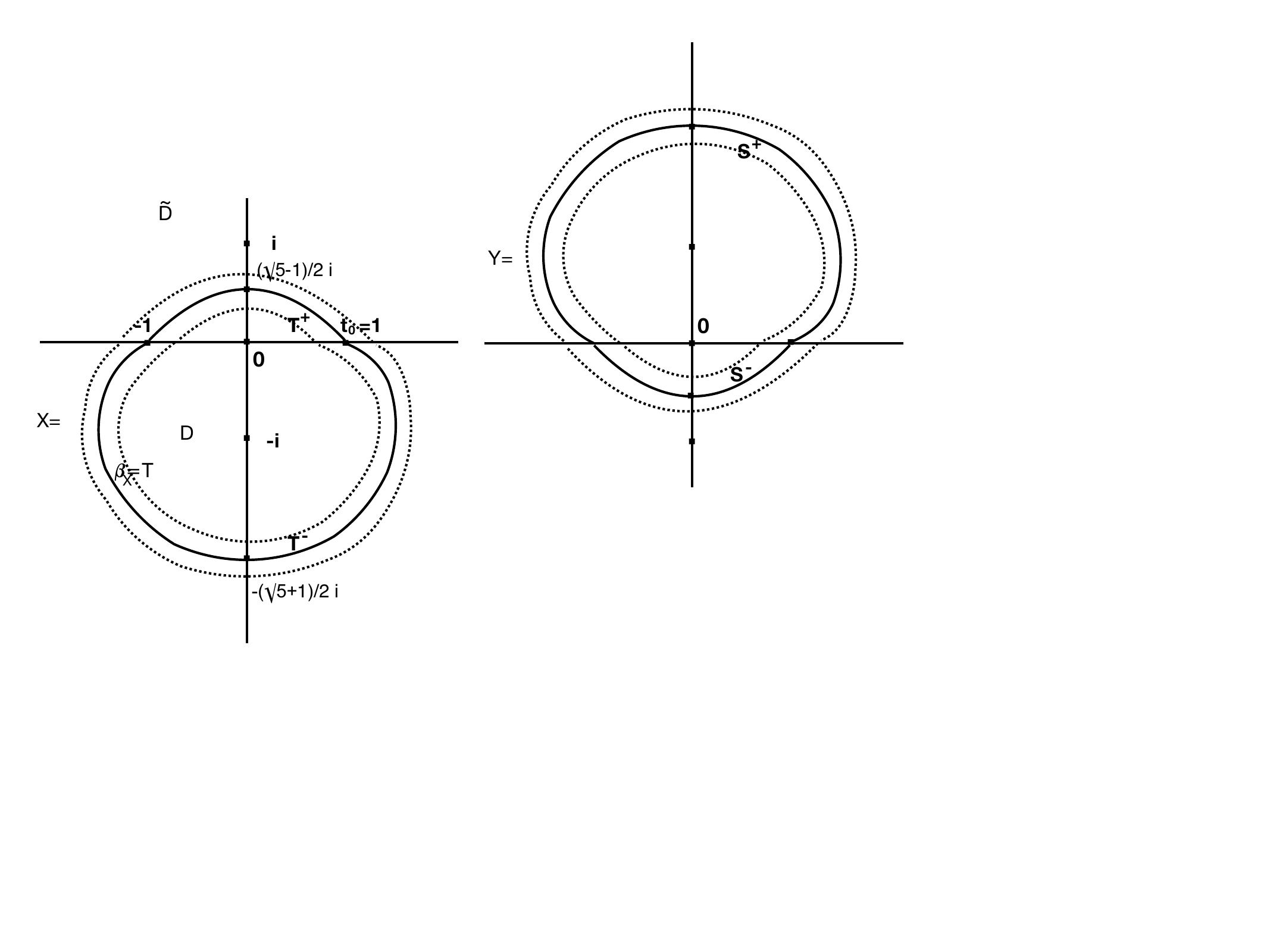}
    \vspace{-1.8in}
  \caption{}
\end{figure}

Let $E=\{ 0, 1, 2, 4, \infty\}\subset \CC$ be a five-point subset in $\CC$. Define
\begin{equation}~\label{exam}
\phi (t, z) =\left\{ \begin{array}{ll} 
                 z& \hbox{if $z=0,1, \infty$ and $t\in X$};\cr
                -\frac{1}{t}+3 & \hbox{if $z=2$ and $t\in X$};\cr
                t +3& \hbox{if $z=4$ and $t\in X$}.                
                \end{array} \right.
\end{equation}
We check that it is a holomorphic motion as follows. 
\begin{itemize}
\item[i)] $\phi (t_{0}, z)=z$ for all $z\in E$; 
\item[ii)] First $-1/t+3$ and $t+3$ will not take $0,1, \infty$ for any $t\in X$ and secondly, since $X$ does not contain $i$ and $-i$, $-1/t\not=t$ for for any $t\in X$. Thus for any fixed $t\in X$, $\phi_{t}(\cdot)=\phi(t,\cdot): E\to \CC$ is injective;
\item[iii)]  Since $X$ does not contain $0$, $-(1/t)+3$ is holomorphic on $X$. Therefore, for any fixed $z\in E$, 
$\phi^{z}(\cdot)=\phi (\cdot, z): X\to \CC$ is holomorphic.  
\end{itemize}
Thus $\phi (t,z)$ is a holomorphic motion of $E$ over $X$.  
The main statement of this paper is 

\medskip
\begin{theorem}[Main Theorem]~\label{main}
The holomorphic motion $\phi$ of $E$ over $X$ defined in (\ref{exam}) satisfies the zero winding number condition but is not fully extendable.
\end{theorem}

\begin{proof}
We first check the zero winding number condition for $\phi$ in (\ref{exam}). 
The set $E$ contains only five points $0,1, 2, 4$, and $\infty$ and $X$ has only one non-trivial simple closed curve $\beta_{X}=T$ in the sense of homotopy. Thus we need only to consider the winding number $\eta(\alpha)$ of 
$$
\alpha =\alpha (\beta_{X}, z_{1}, z_{2})=\phi (\beta_{X}, z_{1})-\phi (\beta_{X}, z_{2})
$$
for any pair of points $z_{1}\not= z_{2}\in E$. 

When $z_{1}, z_{2}=0, 1, \infty$, it is clear that $\eta (\alpha)=0$. 
For $z_{1}=4$ and $z_{2}=0, 1, \infty$, then we have that $\alpha =\beta_{X} +3$, $\beta_{X} +2$, $1/(\beta_{X}+3)$ (for $z_{2}=\infty$, we use the coordinate $w=1/z$). Since $0$ is in the unbounded component of $\C\setminus \alpha$ in these cases, we have $\eta (\alpha)=0$.
For $z_{1}=2$ and $z_{2}=0, 1, \infty$, then we have that $\alpha =-(1/\beta_{X}) +3$, $-(1/\beta_{X}) +2$, $-\beta_{X}+3$ (for $z_{2}=\infty$, we use the coordinate $w=1/z$). In these cases, since $0$ is in the unbounded component of $\C\setminus \alpha$,  we have $\eta (\alpha)=0$. Symmetrically, when $z_{1}=0, 1, \infty$ and $z_{2}=2, 4$, we have also $\eta (\alpha)=0$.

When $z_{1}=4$ and $z_{2}=2$, 
$$
\alpha = \phi (\beta_{X}, 4)-\phi (\beta_{X}, 2) = \beta_{X} +\frac{1}{\beta_{X}}.
$$
It is the image of the simple closed curve $\beta_{X}$ under the map $f(z) =t+1/t$. The map $f(t)$ is a meromorphic function on the simply connected domain $D$ enclosed by $\beta_{X}$  and non-zero and continuous on $\beta_{X}$ (see Figure 1). It has one zero (counted by multiplicity) at $-i$ and one pole (counted by multiplicity) at $0$ insider $D$. From the argument principle in complex analysis, the winding number 
$$
\eta (\alpha) =\frac{1}{2\pi i} \oint_{\beta_{X}} \frac{f'(z)}{f(z)} dz=1-1=0.
$$
Symmetrically, when $z_{1}=2$ and $z_{2}=4$, we have also $\eta(\alpha)=0$. 
We completed the proof of that $\phi$ satisfies the zero winding number condition.

To prove $\phi$ is not fully extendable, we need to introduce the isotopy rel $E$ among all homeomorphisms $h$ of $\CC$.
Two homeomorphisms $h_{0}$ and $h_{1}$ of $\CC$ are said to be isotopic rel $E$, denoted as $h_{1}\sim_{E} h_{0}$,  if there is a continuous map
$$
H(s, z): [0,1]\times \CC\to \CC
$$
such that 
\begin{itemize}
\item[a)] $H(1,z)= h_{0}^{-1}\circ h_{1} (z)$ and $H(0,z)=z$ for all $z\in \CC$;
\item[b)] $h_{s}(\cdot) =H(s, \cdot)$ is a homeomorphism of $\CC$ for every $s\in [0,1]$;
\item[c)] $H(s, z)=z$ for all $z\in E$ and $s\in [0,1]$.
\end{itemize}
We use $[h]_{E}$ to denote the isotopic class of $h$ rel $E$. We call $h$ isotopic trivial rel $E$ if $[h]_{E}=[Id]_{E}$; otherwise, we call $h$ isotopic non-trivial rel $E$. 

Consider the annulus $A=Y+3$ and the simple closed curve $\beta_{Y} =S+3$ which is homotopic to the core curve of $A$. Let $\tau$ be the Dehn twist on $A$ along the curve $\beta_{Y}$ by rotating $2\pi$ clockwisely. 

\begin{figure}[h]
    \includegraphics[width=10in]{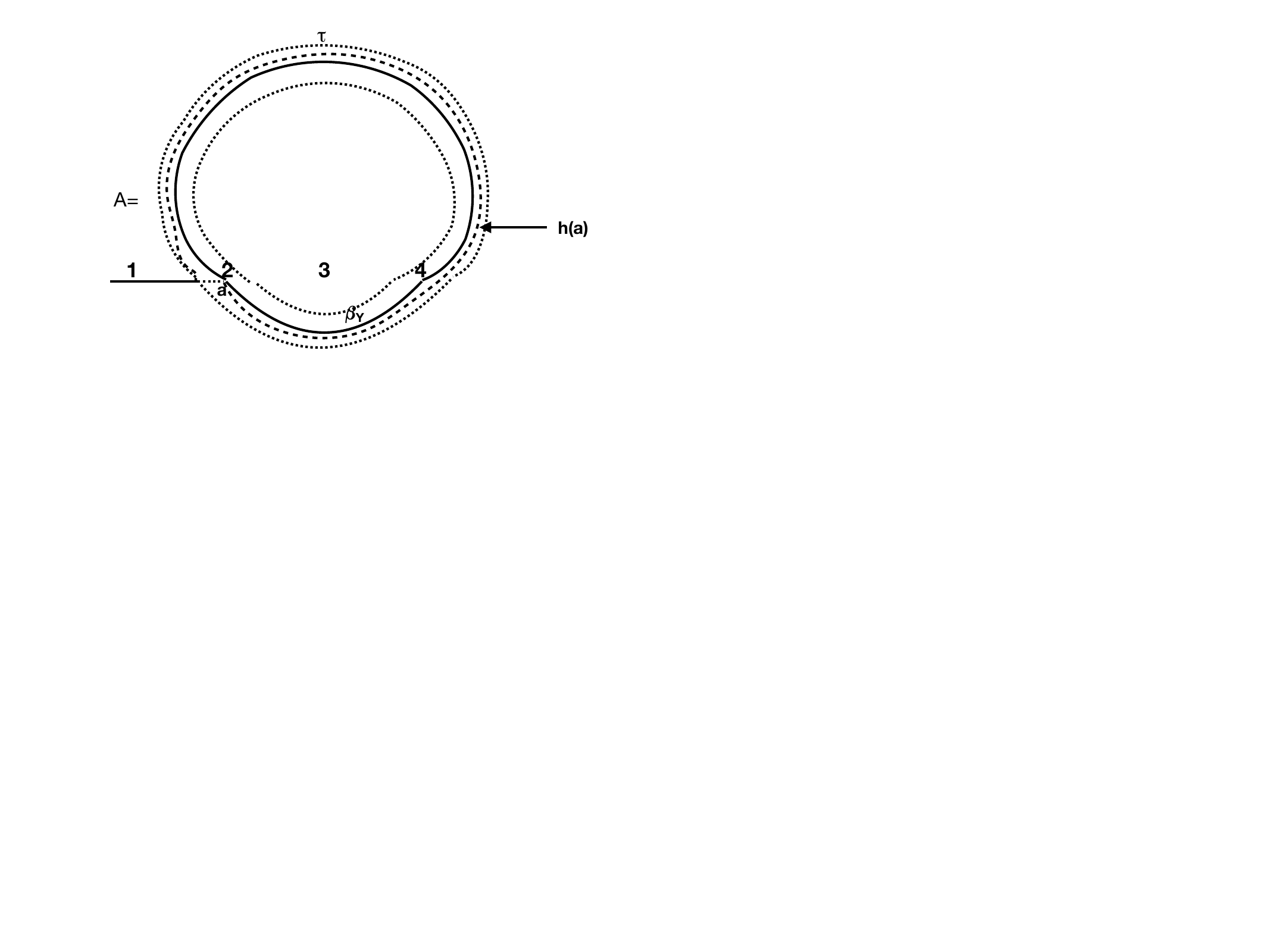}
   \vspace{-4.8in}  
    \caption{}
\end{figure}

\medskip
\begin{lemma}~\label{nh}
Suppose $H$ is a homeomorphism of $\CC$ fixing every point in $E$ such that $H|A$ is $\tau$ in the sense of homotopy (See Figure 2). Then $H$ is isotopic non-trivial rel $E$.
\end{lemma}

\begin{proof}
Consider the Riemann surface $S=\CC\setminus E$. Given two simple closed curves $A$ and $B$ in $S$, 
 the intersection number $I(A, B)$ is by the definition the minimum number among $\#(A'\cap B')$ for all simple closed curves $A'$ homotopic to $A$ in $S$ and all simple closed curves $B'$ homotopic to $B$ in $S$. 
 
 Suppose $H$ is a homeomorphism of $\CC$ fixing every point in $E$.  Suppose $A$ is a simple closed curve in $S$. Let  $B=H(A)$ and consider $I (A, B)$. Then $H$ is isotopic trivial rel $E$ implies that $I(A, B)=0$ for every simple closed curve $A$ in $S$ (refer to~\cite{BC}). We will use this fact in the hyperbolic geometry to show that $H$ in the lemma is isotopic non-trivial rel $E$.    

Let 
$$
A=\{ z\in \C\;|\; \Big| z-\frac{3}{2}\Big|=1\}.
$$
Then $1$ and $2$ are in the bounded component of $\C\setminus A$ and $0$ and $4$ are in the unbounded component of $\C\setminus A$. It is a simple closed curve in $S=\C\setminus E$. Since $H$ restricted to $Y$ is a Dehn twist along $\beta_{Y}$ by rotating $2\pi$ closewisely, $B=h(A)$ is also a simple closed curve in $S$ such that $1$ and $2$ are in the bounded component of $\C\setminus B$ and $0$ and $4$ are in the unbounded component of $\C\setminus B$. (See Figure 3).  However, one can check that the intersection number $I(A, B)=4$. This implies that $H$ must be isotopic non-trivial rel $E$.
\end{proof}

\begin{figure}[h]
    \includegraphics[width=8in]{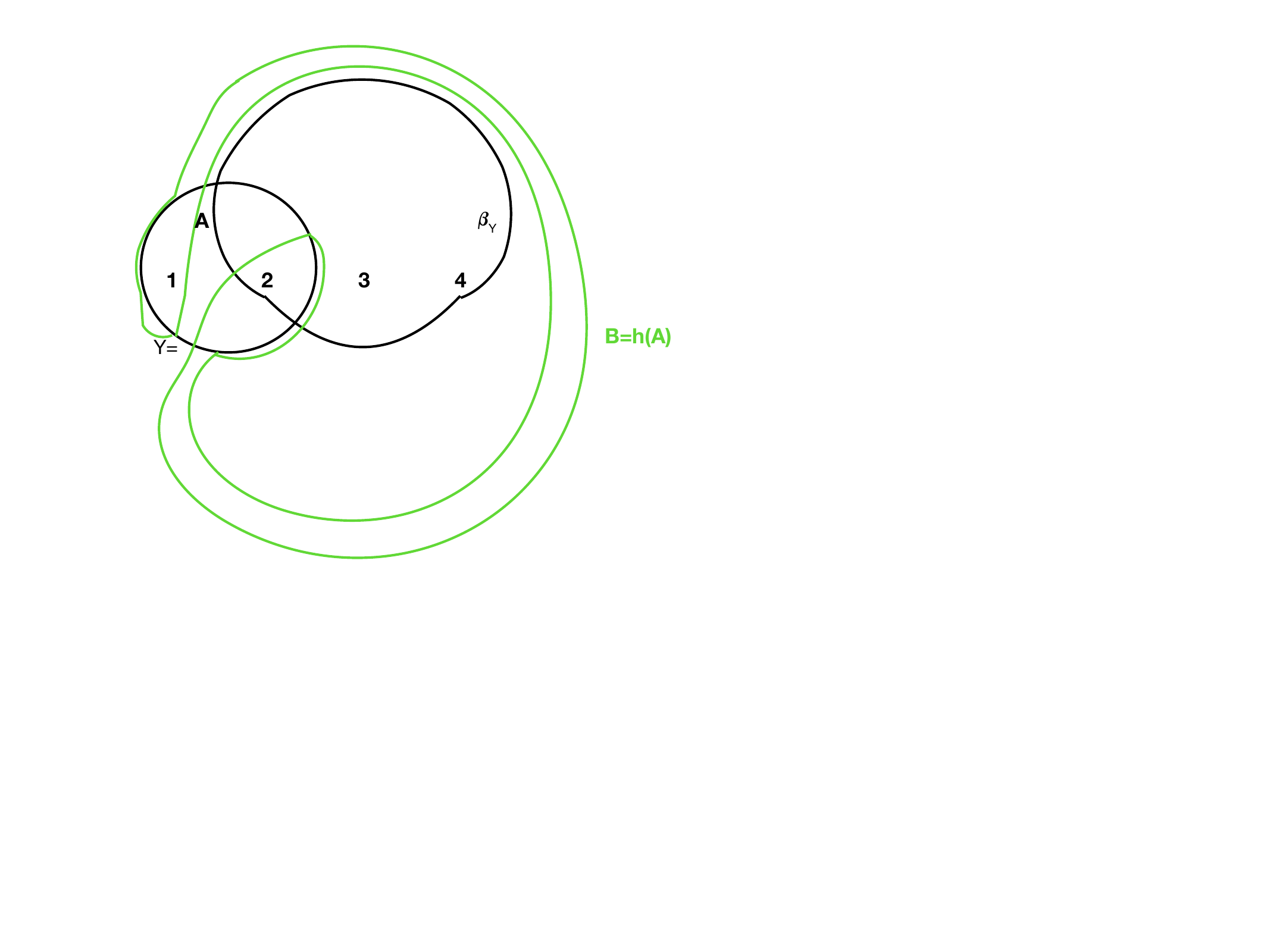}
   \vspace{-2.5in}
  \caption{}
\end{figure}

Let $M(\C)$ be the space of all $L^{\infty}$ complex functions $\mu$ on $\C$ with norms $\|\mu\|_{\infty} <1$, which is a simply connected complex Banach manifold. For every $\mu \in M(\C)$, consider the Beltrami equation
\begin{equation}~\label{be}
w_{\overline{z}} =\mu w_{z}.
\end{equation}
A solution $w$ of (\ref{be}) is called a quasiconformal homeomorphism of $\CC$. If $w(0)=0$, $w(1)=1$, 
and $w(\infty)=\infty$, it is called a normalized solution. 
Here $\mu$ is called the Beltrami coefficient.
The measurable Riemann mapping theorem 
says that for every $\mu\in M(\C)$, the Beltrami equation (\ref{be}) has a unique normalized solution, denoted as $w^{\mu}$. 
Moreover, $w^{\mu}$ depends on $\mu$ holomorphically (refer to~\cite{A,GJW}). 
 Then we can define a normalized holomorphic motion
$$
\Phi (\mu, z) =w^{\mu}(z): M(\C)\times \CC \to \CC.
$$
 
To continue our proof, we consider the fundamental group $\pi_{1} (X) =\mathbb{Z}$, the integer group, 
with the homotopic class $[\beta_{X}]=1$. 
Let $\pi: \Delta\to X$ be the holomorphic universal cover such that $\pi (0)=t_{0}$. 
Let $\Gamma$ be the corresponding group of deck transformations. Let $\gamma$ be the deck transformation 
(a M\"obius transformation on $\Delta$) corresponding to $[\beta_{X}]$, then $\Gamma=<\gamma>$. 
Let $\delta$ be the geodesic connecting $0$ and $\gamma (0)$ in $\Delta$, then
$[\pi (\delta)] =[\beta_{X}]=1$. (See Figure 4).

\begin{figure}[h]
    \includegraphics[width=10in]{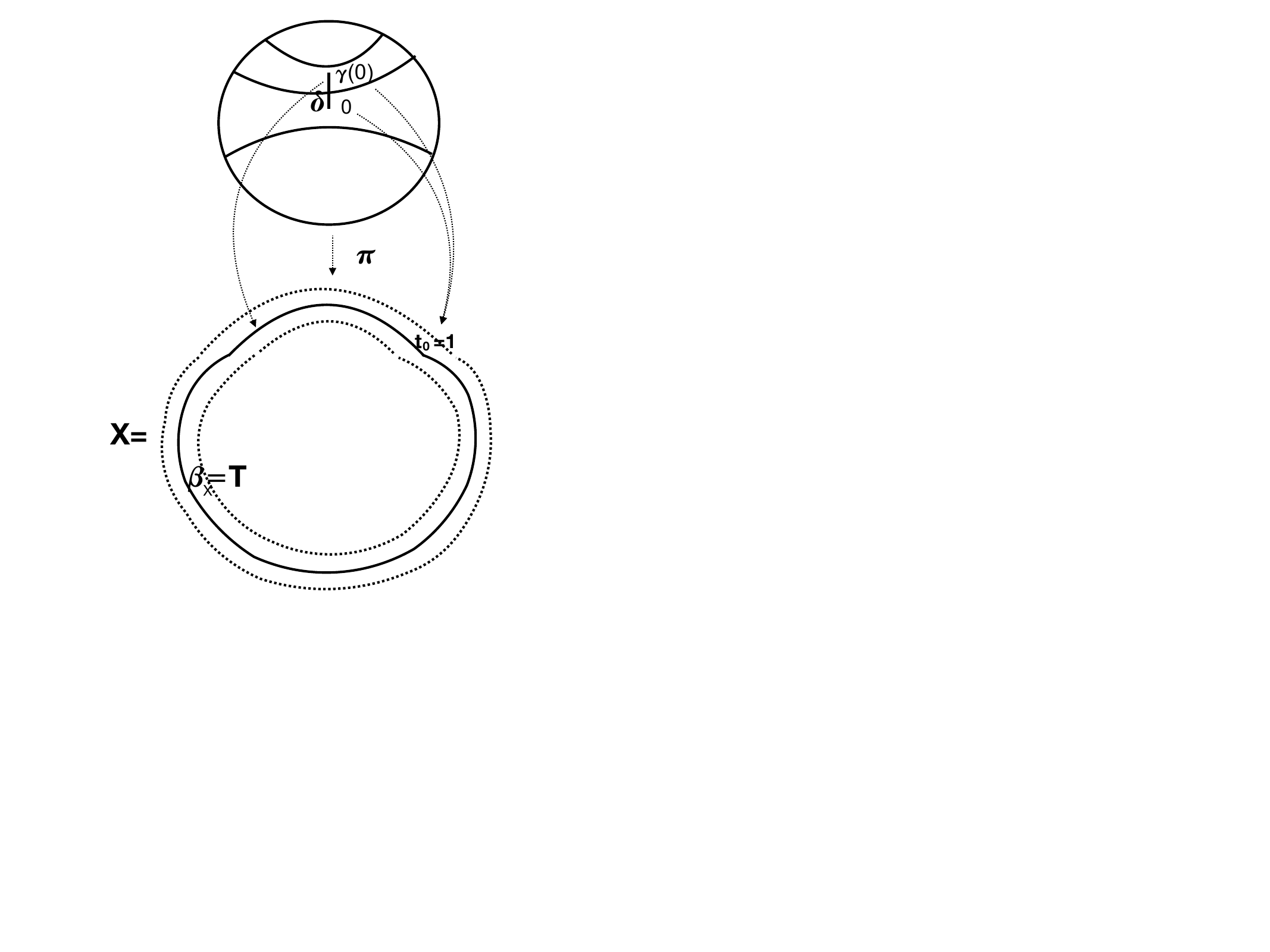}
   \vspace{-3in}
  \caption{}
\end{figure}

Pull back $\phi$ by $\pi$, we have a holomorphic motion 
$$
\widetilde{\phi} (c, z) =\phi (\pi (c), z): \Delta\times E\to \CC,
$$ 
which is fully extendable as we have mentioned in the introduction.
Let $\widetilde{\psi} (c, z): \Delta \times \CC \to \CC$ be a full extension of $\widetilde{\phi}$. For each $c\in \Delta$,
$\widetilde{\psi}_{c} : \CC\to \CC$ is a quasiconformal homeomorphism (refer to~\cite{BR,GJW,M,JM2}). Let $\mu(c)$ be the Beltrami coefficient 
of $\widetilde{\psi}_{c}$. Then 
\begin{equation}~\label{hmap}
F(c)=\mu (c): \Delta \to M(\C), \; F(0)=0,
\end{equation}
defines a holomorphic map (refer to~\cite{BR,GJW,M,JM2}) such that 
\begin{equation}~\label{extension}
\widetilde{\psi} (c, z)= w^{F(c)} (z):  \Delta\times \CC\to \CC
\end{equation} 
is a full extension of $\widetilde{\phi} (c, z): \Delta\times E\to \CC$.

When $c$ moves from $0$ to $\gamma (0)$ along $\delta$ in $\Delta$, $t=\pi (c)$ move the point  
$1$ back to $1$ counter-clockwise on $\beta_{X}$ in $X$ (more precisely, on the unique non-trivial closed geodesic 
homotopic to $\beta_{X}$ in $X$) and  
$$
\widetilde{\psi} (c,2)=\widetilde{\phi} (c, 2) =-\frac{1}{t} +3
$$ 
movers the point $2$ back to the point $2$ clockwisely on 
$\beta_{Y}$ in $Y$. 
Thus we have that

\medskip
\begin{corollary}~\label{con}
The homeomorphism 
$$
H(\cdot) =\widetilde{\psi}(\gamma (0),\cdot)=w^{F(\gamma(0))} (\cdot)$$ 
of $\CC$ fixes every point in $E$ and the restriction $H|A$ is the Dehn twist $\tau$ in the sense of homotopy.
Thus from Lemma~\ref{nh}, $H$ is isotopic non-trivial rel $E$. 
\end{corollary}

We need the next lemma to make sure that the choice of $F$ will not change the isotopic non-trivial property for $H$.   

\medskip
\begin{lemma}~\label{ind}
The isotopic class $[H]_{E}$ does not depend on the choice of any continuous map $F$ in (\ref{hmap}) satisfying (\ref{extension}).
\end{lemma}

\begin{proof} 
Let $F_{1}$ and $F_{2}$ be two maps in (\ref{hmap}) satisfying (\ref{extension}). Parametrize $\delta$ as a path $\delta (s): [0,1]\to \Delta$ from $\gamma (0)$ to $0$. Let
$$
H(s, z)= \Big( w^{F_{1} (\gamma (0))}\Big)^{-1} \circ w^{F_{1}( \delta (s))}\circ \Big( w^{F_{2} (\delta(s))}\Big)^{-1} \circ w^{F_{2} (\gamma (0))} (z) : [0,1]\times \CC\to \CC.
$$   
Then we have that for $s=0$, $\delta (0)=\gamma (0)$, $H(0, z) =z$ for all $z\in \CC$. For $s=1$, $\delta (1)=0$, since $ F_{1}(0)=F_{2}(0)=0$, 
$H(1, z) =(w^{F_{1} (\gamma(0))})^{-1}\circ w^{F_{2}} (\gamma (0))$.  
For $z\in E$ and $s\in [0,1]$, $w^{F_{1}( \delta (s))} (z)=w^{F_{2} (\delta (s))}(z) =\widetilde{\phi} (\delta (s), z)$, we have that $H(s, z)=z$.   
We see that $H(\cdot , \cdot )$ gives an isotopy between $w^{F_{1}( \gamma(0))}$ and $w^{F_{2} (\gamma (0))}$ rel $E$, that is,
$w^{F_{1} (\gamma(0))}\sim_{E} w^{F_{2} (\gamma (0))}$. This completes the proof.
\end{proof}

Now we complete our proof of Theorem~\ref{main} by using the contradiction. That is, we will show that if $\phi$ is fully extendable, then $H(\cdot)=\widetilde{\psi} (\gamma (0), \cdot)$ is isotopic trivial rel $E$. 

Suppose $\phi$ is fully extendable and suppose $\psi (t,z): X\times \CC\to \CC$ is a full extension of $\phi$. For each $t\in X$, $\psi (t,\cdot) : \CC\to \CC$ is a quasiconformal homeomorphism (refer to~\cite{BR,GJW,M,JM2}). Let $\mu (t)$ be the Beltrami coefficient of $\psi (t,\cdot)$. 
Then 
$$
F_{X}(t)=\mu (t): X \to M(\C), \; F_{X}(t_{0})=0,
$$
defines a holomorphic map (refer to~\cite{BR,GJW,M,JM2}) such that 
$$
\psi (t, z)= w^{F_{X}(t)} (z), \; t \in X.
$$
We can now take $F=F_{X}\circ \pi: \Delta\to M(\C)$ as a map in (\ref{hmap}) satisfying (\ref{extension}) such that
$$
H(\cdot)= \widetilde{\psi} (\gamma (0), \cdot) =w^{F(\gamma (0))}(\cdot).
$$ 
This implies that 
$$
[H]_{E}=[w^{F(\gamma (0))}]_{E} =[w^{F_{X} (\pi (\gamma (0)))}]_{E} = [w^{F_{X}(t_{0})}]_{E}=[Id]_{E}.
$$ 
This says that $H$ is isotopic trivial rel $E$. It contracts to Corollary~\ref{con}. The contradiction says that $\phi$ cannot be fully extendable. This completes the proof of Theorem~\ref{main}.
\end{proof}
 
 \medskip
 \medskip 

\bibliographystyle{amsalpha}

\end{document}